\documentclass{article}
\usepackage[a4paper, margin=1in]{geometry}
\usepackage[utf8]{inputenc}
\usepackage{amsmath}
\usepackage{amssymb}
\usepackage{amsthm}
\usepackage{mathtools}
\usepackage{enumerate}
\usepackage{bbm}
\usepackage[capitalise]{cleveref}

\newtheorem{theorem}{Theorem}[section]
\newtheorem{corollary}[theorem]{Corollary}
\newtheorem{lemma}[theorem]{Lemma}

\newtheorem{definition}[theorem]{Definition}

\newcommand{\eps}{\varepsilon}
\DeclareMathOperator{\pc}{PC}
\newcommand{\size}[1]{\lvert{#1}\rvert} 
\newcommand{\Part}{\mathcal{P}}
\newcommand{\chicr}{\chi_{\rm cr}}

\title{Perfect tilings of 3-graphs with the generalised triangle}
\author{Candida Bowtell\thanks{School of Mathematics, University of Birmingham, UK.
{\tt c.bowtell@bham.ac.uk}. Supported by Leverhulme Trust Early Career Fellowship ECF--2023--393.}, 
Amarja Kathapurkar\thanks{School of Mathematics, University of Birmingham, UK. {\tt amarja.kathapurkar@gmail.com}. AK gratefully acknowledges financial support from EPSRC Standard Grant EP/R034389/1.}, 
Natasha Morrison\thanks{Department of Mathematics and Statistics, University of Victoria, Canada. {\tt nmorrison@uvic.ca}. Research supported by NSERC Discovery Grant RGPIN-2021-02511 and NSERC Early Career Supplement DGECR-2021-00047.},
Richard Mycroft\thanks{School of Mathematics, University of Birmingham, UK.
{\tt r.mycroft@bham.ac.uk}. RM is grateful for financial support from EPSRC Standard Grant EP/R034389/1.}}
\date{}

\begin{document}
\maketitle
\begin{abstract}
We establish a best-possible minimum codegree condition for the existence of a perfect tiling of a $3$-uniform hypergraph $H$ with copies of the generalised triangle $T$, which is the 3-uniform hypergraph with five vertices $a, b, c, d, e$ and three edges $abc$, $abd$, $cde$. We also give an asymptotically-optimal minimum codegree condition for the rainbow version of the problem.
\end{abstract}

\section{Introduction}

Over recent years the problem of determining optimal minimum degree conditions which ensure the existence of a perfect tiling of a graph or hypergraph $H$ by copies of a fixed subgraph $F$ has attracted extensive attention from researchers in extremal graph theory. In the graph case, the question was essentially resolved through a series of works by (different subsets of) Alon, Corr\'adi, Hajnal, Koml\'os, K\"uhn, Osthus, S\'ark\"ozy, Szemer\'edi and Yuster~\cite{AY, CH, HSz, Komlos, KSSz, KO2}. 
On the other hand, a general solution in the hypergraph setting remains resolutely out of reach. In this manuscript we make an important step forward towards this goal by determining the optimal minimum codegree condition which guarantees a perfect tiling of a $3$-uniform hypergraph with copies of the generalised triangle. This is the $3$-uniform hypergraph $T$ on five vertices $a, b, c, d, e$ with edges $abc, abd$ and $cde$ (see Figure~1), which has been the subject of significant attention in hypergraph Tur\'an theory. To our knowledge, this is the first non-tripartite 3-uniform hypergraph with more than four vertices for which the optimal minimum codegree threshold for a perfect tiling is known even asymptotically. 

To state our main theorem formally, we make the following definitions. A $k$-uniform hypergraph $H$ (or $k$-graph for short) consists of a set of vertices $V(H)$ and a set of edges $E(H)$, where every edge is a set of $k$ distinct vertices. The degree of a set $f \subseteq V(H)$, denoted $\deg(f)$, is defined to be the number of edges $e \in E(H)$ with $f \subseteq e$, and the minimum codegree of $H$, denoted $\delta(H)$, is the minimum of $\deg(f)$ over all sets $f$ of $k-1$ vertices of $H$. Finally, for a fixed $k$-graph $F$, an $F$-tiling in $H$ is a set $\mathcal{F}$ of vertex-disjoint copies of $F$ in $H$, and $\mathcal{F}$ is perfect if every vertex of $H$ appears in some element of $\mathcal{F}$.

\begin{theorem}[Main result]\label{thm:main}
There exists $n_0 \in \mathbb{N}$ such that for every $n \geq n_0$ with $5 \mid n$, every $3$-uniform hypergraph $H$ on $n$ vertices with $\delta(H) \geq 2n/5$ admits a perfect $T$-tiling.
\end{theorem}

Clearly the condition that $5 \mid n$ is necessary for Theorem~\ref{thm:main}, since the number of vertices of $H$ covered by a $T$-tiling must be a multiple of five. Moreover the minimum degree condition of Theorem~\ref{thm:main} is best possible, as demonstrated by the following example: let $n$ be a multiple of 5, and let $A$ and $B$ be disjoint sets of size $2n/5-1$ and $3n/5+1$ respectively. Let $H_\mathrm{ext}$ be the $3$-graph with vertex set $A \cup B$ whose edges are all triples of vertices which intersect $A$, so $\delta(H) =2n/5-1$. However, since every edge of $H_\mathrm{ext}$ contains a vertex from $A$ and no vertex of $T$ lies in every edge of $T$, we find that every copy of $T$ in $H$ contains at least two vertices from $A$. Consequently every $T$-tiling in $H_\mathrm{ext}$ has size at most $|A|/2 < n/5$, meaning that $H_\mathrm{ext}$ has no perfect $T$-tiling. On the other hand, we believe that the requirement that $n$ is sufficiently large arises solely from the proof techniques, and speculate that in fact Theorem~\ref{thm:main} should hold for every $n \in \mathbb{N}$.
\begin{figure}
\begin{center}
\includegraphics{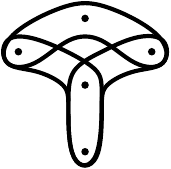}
\end{center}
\caption{The generalised triangle $3$-graph $T$.}
\end{figure}

\subsection{Perfect tilings in graphs and hypergraphs}
The general question of determining whether a graph or hypergraph $H$ contains some given spanning structure -- such as a perfect matching, a Hamilton cycle or a perfect $F$-tiling -- has been one of the most-studied topics in graph theory since the origins of the subject. For example, two classic results in this area are Tutte's theorem~\cite{Tutte} giving a characterisation of all graphs which contain a perfect matching, and Edmonds' algorithm~\cite{Edmonds} which returns a perfect matching in a graph (or reports that no such matching exists) in polynomial time. Unfortunately, we cannot expect similar results for perfect $F$-tilings in general, as Hell and Kirkpatrick~\cite{KH} demonstrated that the problem of determining whether a graph admits a perfect $F$-tiling is NP-hard for every fixed graph $F$ which has a~connected component with at least three vertices. Instead, the central focus of research on this question has been to establish sufficient conditions for the existence of a spanning structure within a graph or hypergraph $G$. Minimum degree conditions on $G$ are the most common parameter to study (since a single isolated vertex in $G$ proscribes the existence of a spanning structure in $G$ with no isolated vertices). For example, Dirac's celebrated theorem~\cite{Dirac} states that every graph $G$ on $n \geq 3$ vertices with minimum degree $\delta(G) \geq n/2$ contains a Hamilton cycle (a cycle which contains every vertex of $G$); for even $n$ this implies that the same condition ensures a perfect matching in $G$. 

This provokes the following general question: for a fixed graph $F$, what is the optimal minimum degree condition on a graph $G$ with $n$ vertices which ensures the existence of a perfect $F$-tiling in $G$? (We assume the necessary condition that the order of $F$ divides the order of $G$, and henceforth assume this without further comment). For $F=K_3$, Corr\'adi and Hajnal~\cite{CH} showed that the best possible sufficient condition is $\delta(G) \geq 2n/3$, before Hajn\'al and Szemer\'edi~\cite{HSz} showed more generally that for $F=K_r$ the condition 
$\delta(G) \geq \frac{r-1}{r}n$ is best possible. Turning to graphs other than cliques, Alon and Yuster~\cite{AY} showed that for every graph $F$ the condition $\delta(G) \geq \tfrac{\chi(F)-1}{\chi(F)}n + o(n)$ is a best-possible condition in terms of the chromatic number $\chi(F)$ up to the $o(n)$ error term. Koml\'os, S\'ark\"ozy and Szemer\'edi~\cite{KSSz} then improved this bound by replacing the $o(n)$ error term with an additive constant (which cannot be removed in general). In the other direction, Koml\'os~\cite{Komlos} introduced the \emph{critical chromatic number} $\chicr(F)$ of~$F$, and observed that for any~$F$ there are graphs $G$ on $n$ vertices with $\delta(G) \geq \tfrac{\chicr(F)-1}{\chicr(F)}n-1$ which do not contain a perfect $F$-tiling. Finally K\"uhn and Osthus~\cite{KO2} gave the correct minimum degree condition up to an additive constant $c$ for every graph $F$, which is either $\delta(G) = \tfrac{\chicr(F)-1}{\chicr(F)}n + c$ or $\delta(G) = \tfrac{\chi(F)-1}{\chi(F)}n + c$, according to the value of the \emph{greatest common divisor} of $F$. A survey by K\"uhn and Osthus~\cite{KO} gives a more detailed overview.

Together, these results are essentially a full resolution of the question in the graph setting. By contrast, results for the analogous question in $k$-graphs with $k \geq 3$ are limited for the most part to the special case where $F$ is a $k$-partite $k$-graph. For matchings (where $F$ is a single edge) R\"odl, Ruci\'nski and Szemer\'edi~\cite{RRS} established for large $n$ the precise best-possible minimum codegree condition for a perfect matching in $H$, which is $\delta(H) \geq n/2 - c$ where $c \in \{1/2, 1, 3/2, 2\}$ depends on the parities of quantities involving $k$ and $n$. This improved on successively-stronger previous approximate bounds given by R\"odl, Ruci\'nski and Szemer\'edi~\cite{RRS2}, K\"uhn and Osthus~\cite{ko_match} and R\"odl, Rucinski, Schacht and Szemer\'edi~\cite{rrss}. For all $k$-partite $k$-graphs $F$ Mycroft~\cite{Mycroft} gave a sufficient minimum codegree condition for the existence of a perfect $F$-tiling in a $k$-graph $H$ on $n$ vertices, and showed that this condition is aysmptotically best-possible for a large class of $k$-partite $k$-graphs including every complete $k$-partite $k$-graph. In the specific case where $F$ is a 3-graph with four vertices and two edges this had already been established by K\"uhn and Osthus~\cite{ko_loose}, and a best-possible exact bound for large $n$ was then given by Czygrinow, DeBiasio, and Nagle~\cite{cdn}. Gao, Han and Zhao~\cite{GHZ} did the same for the $k$-graph with $k+1$ vertices and two edges for each $k$, as well as giving a stronger bound on the sublinear error term in Mycroft's result, and giving a best-possible exact bound in the case where $F$ is a loose cycle (which is a particular form of $k$-partite $k$-graph).

Much less is known for non-$k$-partite $k$-graphs. For $F = K^3_4$ Lo and Markstr\"om~\cite{LM} showed that every $3$-graph $H$ on $n$ vertices with $\delta(H) \geq 3n/4 + o(n)$ admits a perfect $K^3_4$-tiling, which is asymptotically best-possible. Keevash and Mycroft~\cite{KMy} simultaneously and independently gave a best-possible exact minimum codegree condition for a perfect $K^3_4$-tiling in $H$ for large $n$, which is $\delta(H) \geq 3n/4 - 2$ if $8 \mid n$ and $\delta(H) \geq 3n/4 - 1$ otherwise. These works improved on weaker sufficient conditions given by Czygrinow and Nagle~\cite{cn} and Pikhurko~\cite{pikhurko}. Also, for $F$ being the unique $3$-graph $K^3_4-e$ on $4$ vertices with 3 edges, Lo and Markstr\"om~\cite{lm2} showed that $\delta(H) \geq n/2 + o(n)$ is an asymptotically best-possible minimum codegree condition for a perfect $K_4^3-e$ tiling in $H$, before Han, Lo, Treglown and Zhao~\cite{HLTZ} showed that for large $n$ the best-possible condition is precisely $\delta(H) \geq n/2 - 1$. Finally, Han, Lo and Sanhueza-Matamala~\cite{hlsm} gave an asymptotically best-possible minimum codegree condition for an $F$-tiling in a $k$-graph $H$ in the case when $k \geq 4$ is even and $F$ is a long tight cycle with certain divisibility properties. To our knowledge $K^3_4$ and $K^3_4-e$ are the only two non-tripartite $3$-graphs $F$ for which the optimal minimum codegree condition for a perfect $F$-tiling in a $k$-graph $H$ was known even asymptotically prior to the work presented in this manuscript.

For a wider overview of research on these topics we recommend the surveys of K\"uhn and Osthus~\cite{KO}, R\"odl and Ruci\'nski~\cite{RR} and Zhao~\cite{Zhao}.

\subsection{The generalised triangle}

The generalised triangle $3$-graph $T$ which is the focus of this manuscript is considered by many to be the most natural $3$-graph analogue of the (graph) triangle and consequently has been the focus of significant attention. We summarise these results here, categorised by the analogous results in the graph case.

\begin{itemize}
    \item Mantel's theorem~\cite{mantel} identifies the extremal number of the triangle graph $K_3$, showing that the maximum possible number of edges in a triangle-free graph on $n$ vertices is uniquely achieved by the complete balanced bipartite graph, which has $\lfloor n^2/4 \rfloor$ edges. For 3-graphs, Frankl and F\"uredi~\cite{FF} proved that for $n \geq 3000$ no $T$-free $3$-graph on $n$ vertices has more edges than the complete balanced tripartite graph, a result described by Keevash~\cite{keevash} as ``the first hypergraph Tur\'an theorem". Keevash and Mubayi~\cite{KMu} subsequently showed that the same holds for $n \geq 33$ and also that the complete balanced tripartite graph is the unique $T$-free $3$-graph with this many edges, before Goldwasser~\cite{Goldwasser} found the maximum number of edges in a $T$-free $3$-graph on $n$ vertices, including a characterisation of the extremal $3$-graphs, for every natural number $n$. These results built on previous work of Bollobas~\cite{Bollobas}, which answered a question of Katona, confirming that the complete balanced tripartite 3-graph is unique with the most edges among cancellative $3$-graphs, which are $3$-graphs which do not contain a copy of either $T$ or $K^3_4-e$. 
    \item The Erd\H{o}s-Kleitman-Rothschild~\cite{EKR} theorem states that almost all triangle-free graphs on $n$ vertices are bipartite. Balogh and Mubayi~\cite{BM} gave the natural $3$-graph analogue, showing that almost all $T$-free $3$-graphs on $n$ vertices are tripartite. Balogh, Butterfield, Hu and Lenz~\cite{BBHL} gave a sparse version of this result, namely that for $p \geq K \log n/n$ the binomial random $3$-graph $H$ on $n$ vertices has with high probability the property that every maximum $T$-free subgraph of $H$ is tripartite; they also showed that the same does not hold for $p = 0.1 \sqrt{\log n}/n$.
    \item Thomassen~\cite{Thomassen} showed that the chromatic threshold of triangle-free graphs is $1/3$, that is, that for every $c > 1/3$ there exists $k$ such that every triangle-free graph $G$ with minimum degree $\delta(G) \geq c|V(G)|$ has chromatic number $\chi(G) \leq k$ (and the same is not true if $1/3$ is replaced by any smaller real number). Balogh, Butterfield, Hu, Lenz and Mubayi~\cite{BBHLM} studied the analogous chromatic threshold of $T$, showing that its value lies between $6/49$ and $(\sqrt{41}-5)/8$.
    \item Andrasfai, Erd\H{o}s and S\'os~\cite{AES} proved that every triangle-free graph $G$ on $n$ vertices with minimum degree $\delta(G) \geq 2n/5$ is bipartite. Analogous statements for $3$-graphs were given by Liu, Ren and Wang~\cite{LRW}, who showed that every $T$-free $3$-graph $H$ on $n \geq 5000$ vertices with minimum vertex degree $\delta_1(H) \geq 4n^2/45$ is tripartite and (in a separate paper) that every $T$-free $3$-graph $H$ on $n \geq 5$ vertices with minimum positive codegree $\delta^+(H) \geq 2n/5$ is tripartite. (For a $3$-graph $H$, the minimum vertex degree $\delta_1(H)$ of $H$ is the minimum of $\deg(v)$ over all vertices $v \in V(H)$, whilst the minimum positive codegree $\delta^+(H)$ of $H$ is the smallest $\delta$ such that every pair $x, y \in V(H)$ with $\deg(xy) \geq 1$ has $\deg(xy) \geq \delta$.)
    \item It is trivial to show that if $G$ is a graph on $n$ vertices with $\delta(G) > n/2$ then each vertex of $G$ is covered by some triangle in $G$, and the complete bipartite graph demonstrates that the minimum degree bound of this statement cannot be improved. Gu and Wang~\cite{GW} observed the analogous covering codegree number of $T$ and gave bounds on the covering vertex degree number of $T$; the latter value was subsequently determined asymptotically by Ma, Hou and Yin~\cite{MHY}.
\end{itemize}

In a similar vein, it seems natural to regard the main result of this paper as a $3$-graph analogue of the Corr\'adi-Hajnal theorem~\cite{CH}, which states that every graph $G$ on $n$ vertices with minimum degree $\delta(G) \geq 2n/3$ admits a perfect triangle-tiling.

\subsection{Proof strategy}
Our proof of Theorem~\ref{thm:main} proceeds by distinguishing an `extremal' case, in which $H$ is close to the example $H_\mathrm{ext}$ presented after the statement of Theorem~\ref{thm:main}, and a `non-extremal' case covering all other possibilities. We formalise this distinction by the following definition. Note for this that the density of a $k$-graph $H$ on $n$ vertices is $d(H) := e(H)/\binom{n}{k}$, where $e(H)$ denotes the number of edges of $H$. Also, for a set $S \subseteq V(H)$ the subgraph of $H$ induced by $S$, denoted $H[S]$, is the $k$-graph with vertex set $S$ whose edges are all edges $e \in E(H)$ with $e \subseteq S$. 

\begin{definition}
Let $H$ be a $3$-graph on $n$ vertices. For $\gamma > 0$ we say that $H$ is $\gamma$-extremal if there exists a set $S \subseteq V(H)$ with $|S| = \lfloor3n/5\rfloor$ such that $H[S]$ has density at most $\gamma$. 
\end{definition}

The following lemma states that Theorem~\ref{thm:main} holds in the extremal case. The proof is in Section~\ref{sec:extremal}; the approach is specific to the generalised triangle $T$, with the key step being to find a perfect matching in an appropriate auxiliary graph. We note that the precise minimum degree condition $\delta(H) \geq 2n/5$ is needed for the proof of this lemma; all other components of the proof still hold under a slightly weaker degree condition.

\begin{lemma}[Extremal case] \label{lem:extremal}
There exist $\gamma > 0$ and $n_0 \in \mathbb{N}$ for which the following holds. Let $H$ be a $3$-graph on $n \geq n_0$ vertices with $\delta(H) \geq 2n/5$ and $5 \mid n$. If $H$ is $\gamma$-extremal, then $H$ contains a perfect $T$-tiling.
\end{lemma}

It remains to prove Theorem~\ref{thm:main} in the non-extremal case. For this we use an absorbing approach, which has two main components. The first is the following absorbing lemma, which asserts the existence of a small $T$-tiling which can `absorb' into itself any small set of vertices (subject to the necessary condition that the number of vertices to be absorbed is a multiple of 5). We prove this lemma in Section~\ref{sec:absorbing}. Note that we don't require the assumption that $H$ is non-extremal for the absorbing lemma. Furthermore, the minimum codegree condition required for Lemma~\ref{lem:absorbing} is significantly weaker than the minimum codegree condition of Theorem~\ref{thm:main}.

\begin{lemma}[Absorbing lemma] \label{lem:absorbing} 
Suppose that $1/n \ll c \ll \alpha$ and let $H$ be a $3$-graph on $n$ vertices. If $\delta(H) \geq n/3+\alpha n$, then there exists a set $A \subseteq V(H)$ with $|A| \leq cn$ such that for every $S \subseteq V(H) \setminus A$ with $|S| \leq c^2n$ and $5 \mid |S|$ there exists a perfect $T$-tiling in $H[A \cup S]$.
\end{lemma}

The second component of the proof in the non-extremal case is the following tiling lemma, which states that a minimum codegree condition close to that of Theorem~\ref{thm:main} is sufficient to ensure a $T$-tiling in $H$ covering almost all of the vertices of $H$, under the additional assumption that $H$ is not extremal. We prove this lemma in Section~\ref{sec:tiling}.

\begin{lemma}[Tiling lemma] \label{lem:tiling}
Suppose that $1/n \ll \alpha, \eta \ll \gamma$, and that $H$ is a $3$-graph on $n$ vertices. If $\delta(H) \geq 2n/5 - \alpha n$,
then either $H$ admits a $T$-tiling covering at least $(1-\eta) n$ vertices of $H$, or $H$ is $\gamma$-extremal.
\end{lemma}

Our proof of the tiling lemma uses Farkas' lemma to find a perfect fractional matching in $H$, but our application of Farkas' lemma is dependent on the structure of $H$, with different cases according to the behaviour of the neighbourhoods of certain vertices. To our knowledge this approach is novel, and as discussed further in Section~\ref{sec:tiling} we hope that it could potentially be developed to prove further results on perfect tilings in hypergraphs.

Combining Lemmas~\ref{lem:extremal},~\ref{lem:absorbing} and~\ref{lem:tiling} yields Theorem~\ref{thm:main}.

\begin{proof}[Proof of Theorem~\ref{thm:main}]
Introduce constants with $1/n_0 \ll c, \alpha \ll \gamma \ll 1$. Let $n \geq n_0$ be divisible by 5, and let $H$ be a $3$-graph on $n$ vertices with $\delta(H) \geq 2n/5$. By Lemma~\ref{lem:absorbing} there exists a set $A \subseteq V(H)$ with $|A| \leq cn$ such that for every $S \subseteq V(H) \setminus A$ with $|S| \leq c^2n$ and $5 \mid |S|$ there exists a perfect $T$-tiling in $H[A \cup S]$. In particular $H[A]$ admits a perfect $T$-tiling, so we must have $5 \mid |A|$. Let $V' := V(H) \setminus A$, let $H' = H[V']$ and let $n' = |V'|$, so $n-cn \leq n' \leq n$ and $5 \mid n'$, and also $\delta(H') \geq 2n/5 - cn \geq 2n'/5 - \alpha n'$. By Lemma~\ref{lem:tiling} (applied with $c^2$ and $\gamma/2$ in place of $\eta$ and $\gamma$ respectively) it follows that either $H'$ admits a $T$-tiling covering at least $(1-c^2)n'$ vertices of $H'$, or $H'$ is $\gamma/2$-extremal.

Suppose first that $H'$ admits a $T$-tiling $\mathcal{T}$ which covers at least $(1-c^2)n'$ vertices of $H'$. Then $S := V' \setminus V(\mathcal{T})$ is a subset of $V(H) \setminus A$ with $|S| \leq c^2n' \leq c^2n$, and moreover $5 \mid |S|$ since both $n'$ and $|V(\mathcal{T})|$ are divisible by 5. By our choice of $A$ it follows that $H[A \cup S]$ admits a perfect $T$-tiling $\mathcal{T}'$, whereupon $\mathcal{T} \cup \mathcal{T}'$ is a perfect $T$-tiling in $H$, as required.

Now suppose instead that $H'$ is $\gamma/2$-extremal, meaning that there exists a set $S' \subseteq V(H')$ with $|S'| = 3n'/5$ such that $d(H[S']) \leq \gamma/2$.
By adding an arbitrary $3(n-n')/5 \leq cn$ vertices of $V(H) \setminus S'$ to $S'$ we obtain a set $S \subseteq V(H)$ with $|S| = 3n/5$ such that $d(H[S]) \leq ((\gamma/2) \binom{n'}{3} + cn \binom{n}{2})/\binom{n}{3} \leq \gamma$.
This set $S$ witnesses that $H$ is $\gamma$-extremal, and so $H$ admits a perfect $T$-tiling by Lemma~\ref{lem:extremal}, completing the proof.
\end{proof}

\subsection{Rainbow tilings}

We also consider minimum codegree conditions for rainbow $T$-tilings in families of 3-graphs. Let $V$ be a set of $n$ vertices, and let $\mathcal{H} = \{H_1, \dots, H_{3n/5}\}$ be a family of $3$-graphs with common vertex set~$V$. A {\it perfect rainbow $T$-tiling} in $\mathcal{H}$ is a perfect $T$-tiling in the multi-$3$-graph $\bigcup_{i \in [3n/5]} H_i$ consisting of edges $e_1, e_2, \ldots, e_{3n/5}$ such that $e_i \in E(H_i)$ for each $i \in [3n/5]$. In Section~\ref{sec:rainbow} we give a short argument combining Theorem~\ref{thm:main} with a recent theorem given by Lang~\cite{Lang} to derive the following corollary.

\begin{theorem}\label{thm:rainbow}
    For all $\eps > 0$ there exists $n_0$ such that the following holds for every $n \geq n_0$ with $5 \mid n$. Let $V$ be a set of $n$ vertices and let $\mathcal{H}=\{H_1, H_2, \ldots, H_{3n/5}\}$ be a collection of $3$-uniform hypergraphs with common vertex set $V$. If $\delta(H_i) \geq (\frac{2}{5}+\eps)n$ for each $i \in [\frac{3n}{5}]$, then $\mathcal{H}$ admits a perfect rainbow $T$-tiling.
\end{theorem}

The minimum codegree bound of Theorem~\ref{thm:rainbow} is best possible up to the error term $\eps n$, as can be seen by taking $H_1 = H_2 = \dots = H_{3n/5} = H_\mathrm{ext}$, where $H_\mathrm{ext}$ was the 3-graph demonstrating the optimality of Theorem~\ref{thm:main}. We expect that with further work it should be possible to improve Theorem~\ref{thm:rainbow} by replacing the condition $\delta(H_i) \geq (\frac{2}{5}+\eps)n$ with the condition $\delta(H_i) \geq \frac{2}{5}n + c$ for a fixed constant~$c$. However, our argument cannot be directly extended in this way: the result proved by Lang on which we rely  deals only with asymptotic bounds on the minimum degree rather than exact bounds.

\subsection{Notation}

We use the following forms of notation in this manuscript, all of which are standard within this area. For $n \in \mathbb{N}$ we write $[n] = \{1, 2, \dots, n\}$, and for a set $X$ we write $\binom{X}{n}$ for the set of all subsets of $X$ of size $n$. We write $x \ll y$ to mean that for all $y > 0$ there exists $x_0 > 0$ such that for every $x > 0 $ with $x < x_0$ the subsequent statements hold, and define similar statements with more variables in the analogous way.

Let $H$ be a $k$-graph. For a set of vertices $f \subseteq V(H)$ the neighbourhood of $f$ is $N(f) := \{g \subseteq V(H) : f \cup g \in E(H)\}$. So $\deg(f) = |N(f)|$. Where necessary to avoid ambiguity we write $N_H(f)$ and $\deg_H(f)$ to indicate that we mean the neighbourhood or degree within the $k$-graph $H$. Also, if $H$ is a $3$-graph then for notational convenience we usually write $N(xy)$ and $\deg(xy)$ rather than $N(\{x, y\})$ and $\deg(\{x, y\})$. Observe that in this case $N(xy)$ is a set of vertices; we refer to these as neighbours of $x$ and $y$.

A $k$-graph $H$ is $k$-partite if there exists a partition of $V(H)$ into $k$ parts for which every edge of $H$ has precisely one vertex in each part. For $k=3$ we say tripartite rather than 3-partite.

We say that a set $S$ of 5 vertices in a 3-graph $H$ supports a copy of $T$ if $H[S]$ contains a copy of $T$.

\section{The extremal case} \label{sec:extremal}
The aim of this section is to prove Lemma~\ref{lem:extremal}, that is, that Theorem~\ref{thm:main} holds in the extremal case. Our argument will complete the desired $T$-tiling by finding a perfect matching in a dense auxiliary $5$-partite $5$-graph, using a corollary of a theorem of Daykin and H\"aggkvist~\cite{DH}.

\begin{theorem}[Daykin and H\"aggkvist~\cite{DH}] \label{5partitematching}
Let $J$ be a $k$-partite $k$-graph whose vertex classes each have size $n$. If every vertex of $J$ is contained in at least $(k-1)n^{k-1}/k$ edges of $J$, then $J$ admits a perfect matching.
\end{theorem}

\begin{corollary} \label{auxiliarymatching}
Suppose that $1/n \ll \beta \ll 1$. Let $A$ and $B$ be disjoint sets of size $3n$ and $2n$ respectively, and let $J$ be a $5$-graph on~$5n$ vertices with vertex set $V := A \cup B$ in which every edge contains 3 vertices of $A$ and 2 vertices of $B$. If $e(J) \geq (1-\beta) \binom{3n}{3}\binom{2n}{2}$ and every vertex of $J$ is contained in at least $n^4/10^{8}$ edges of $J$, then $J$ admits a perfect matching.  
\end{corollary}

\begin{proof}
Let $G$ be the $5$-graph with vertex set $V$ whose edges are all $5$-tuples in $V$ which contain $3$ vertices from $A$ and $2$ vertices from $B$. Observe that $e(G)=\binom{3n}{3}\binom{2n}{2}$ and $9n^4 - 9n^3 \leq \deg_G(v) \leq 9n^4 $ for every $v \in V$. Also $J \subseteq G$ is a spanning subgraph of $G$. Let $X \subseteq V$ consist of all vertices $x$ for which $\deg_G(x) - \deg_J(x) \geq \beta^{1/2}n^4$.
As $e(J) \geq (1-\beta) \binom{3n}{3}\binom{2n}{2}$, we must have $|X| \le 50\beta^{1/2}n$.
Observe that removing a collection of $m$ vertices from $J$ will decrease the degree of each remaining vertex by at most $27mn^3$. As $|X| \le 50\beta^{1/2}n$ and every vertex is in at least $n^4/10^{8}$ edges, we can thus greedily cover the vertices of $X$ 
by a matching $M_1$ in $J$ of size $|M_1| \leq |X|$.

Let $V_1 := V \setminus V(M_1)$, and define $J_1 := J[V_1]$ and $G_1 := G[V_1]$, so $J_1 \subseteq G_1$. 
For each $v \in V_1$ we then have 
$$\deg_{J_1}(v) \geq \deg_J(v)- 27|V(M_1)|n^3 \geq \deg_J(v) - 135 |X| n^3 \geq \deg_J(v) - 10000\beta^{1/2}n^4$$ 
and $\deg_{G_1}(v) \leq \deg_{G}(v)$; since $v \notin X$ it follows that $\deg_{G_1}(v) - \deg_{J_1}(v) \le \beta^{1/3}n^4$. Now 
arbitrarily partition $A'=A\setminus V_1$ into $3$ parts of equal size, and $B'=B\setminus V_1$ into two parts of equal size and let $J_2 \subseteq J_1$ and $G_2 \subseteq G_1$ be the subgraphs with vertex set~$V_1$ whose edge sets are all edges of $J_1$ and $G_1$ (respectively) which contain exactly one vertex in each of the $5$ parts. Observe that $J_2$ and $G_2$ are $5$-partite $5$-graphs with common vertex classes each of size $n_2=n-|M_1| \leq n$. Moreover for every $v \in V_1$ we have $\deg_{G_2}(v) - \deg_{J_2}(v) \leq \deg_{G_1}(v) - \deg_{J_1}(v) \leq \beta^{1/3}n^4$ and $\deg_{G_2}(v) = n_2^4$, so $\deg_{J_2}(v) \geq n_2^4-\beta^{1/3}n^4 \geq 4n_2^4/5$. Hence by Theorem~\ref{5partitematching} $J_2$ admits a perfect matching $M_2$. Together with $M_1$ this yields a perfect matching in~$J$.
\end{proof}

We can now give the proof of Lemma~\ref{lem:extremal}.

\begin{proof}[Proof of Lemma~\ref{lem:extremal}]
Suppose that $1/n \ll \gamma \ll \gamma'  \ll \beta \ll 1$ and that $5 \mid n$. Let $H$ be a 3-graph on $n$ vertices with $\delta(H) \geq 2n/5$ which is $\gamma$-extremal. This means there is a set $S \subseteq V(H)$ with $|S| = 3n/5$ such that $d(H[S]) \leq \gamma$. Say that a pair $x, y$ of vertices of $H$ is bad if $x, y \in S$ and $|N(xy) \cap S| > \sqrt{\gamma} n$, and good if $x, y \in S$ and $|N(xy) \cap S| \leq \sqrt{\gamma} n$ (pairs not contained in $S$ are neither good nor bad).
There are then at most $\sqrt{\gamma} n^2$ bad pairs, as otherwise we would obtain more than $1/3 \cdot (\sqrt \gamma n^2) \cdot (\sqrt{\gamma} n) \geq \gamma \binom{n}{3}$ edges in $H[S]$, contradicting our bound on the density of $H[S]$. Let~$X$ be the set of vertices in $V(H) \setminus S$ which are in fewer than $n^2/50$ edges with precisely 2 vertices in~$S$. Since there are at most $\sqrt{\gamma} n^2$ bad pairs, and every good pair has at least $2n/5-\sqrt{\gamma} n$ neighbours in $V(H) \setminus S$, the number of triples in $V(H)$ with precisely two vertices in $S$ which are not edges of $H$ is at most $(\sqrt{\gamma} n^2) \cdot (2n/5) + \binom{3n/5}{2} \cdot (\sqrt{\gamma} n) \leq \sqrt{\gamma} n^3$. Since each vertex in $X$ is in at least $\binom{3n/5}{2} - n^2/50 \geq n^2/10$ such triples, it follows that $|X| \leq 10\sqrt{\gamma} n$.  
Define $A = S \cup X$ and $B = V(H) \setminus (S \cup X)$ and note that $3n/5 \leq |A| = 3n/5+|X| \leq (3/5+10 \sqrt{\gamma})n$, and $(2/5-10 \sqrt{\gamma})n \leq |B|=2n/5-|X| \leq 2n/5$. Furthermore, each good pair is contained in $A$ and  has at least $2n/5-11\sqrt{\gamma} n$ neighbours in~$B$.

We claim that there exists a matching $M$ in $H[A]$ of size $|X|$ such that each edge of $M$ contains a good pair. To see this consider a largest matching $M$ in $H[A]$ with $|M| \leq |X|$ in which each edge contains a good pair. If $|M| < |X|$, then $|V(M)| \leq 3|X| - 3 \leq 30 \sqrt{\gamma} n$, so we may choose three pairwise-disjoint good pairs $p_1, p_2, p_3$ which do not intersect $V(M)$. By the minimum codegree condition, each of these pairs has at least $|X|$ neighbours in $A$. Each of these neighbours must be covered by $M$, as otherwise we would be able to extend $M$ to a larger matching in $H[A]$ in which each edge contains a good pair. Since $|V(M)| \leq 3|X| - 3$ but $\sum_{i \in [3]} |N(p_i) \cap A| \geq 3|X|$ it follows that there is an edge $e \in M$, vertices $x, y \in e$ and $i, j \in [3]$ with $x \neq y$ and $i \neq j$ for which $x \in N(p_i)$ and $y \in N(p_j)$.  This means we can replace $e$ in $M$ by $p_i \cup x$ and $p_j \cup y$ to obtain a larger matching in $H[A]$ in which each edge contains a good pair. We conclude that in fact $|M| = |X|$, giving the required matching.

We now use $M$ to construct a set of $|X|$ pairwise vertex-disjoint copies of $T$ in $H$ which each have four vertices in $A$ and one vertex in $B$. For this, consider an edge $uvw$ of $M$, where $uv$ is a good pair, and choose a vertex $z \in A \setminus V(M)$ which has not been included in a previously-chosen copy of $T$. If $zw$ has at least $\gamma'n$ neighbours in $B$, then $uv$ and $wz$ have at least $\gamma' n - 11 \sqrt{\gamma} n \geq 5\size{X} +|V(M)|$ common neighbours in $B$, so we can choose a previously-unused common neighbour $y$ of $uv$ and $zw$ in $B \setminus V(M)$ to obtain the desired copy of $T$ with edges $uvy$, $uvw$ and $zwy$. If not, then $zw$ has at least $2n/5 - \gamma' n$ neighbours in $A$, so we can choose neighbours $a, b$ of $zw$ in $A \setminus V(M)$ which have not previously been used and which form a good pair, meaning that we can choose a previously-unused neighbour $c$ of $ab$ in $B$ to form the desired copy of $T$ with edges $zwa, zwb$ and $abc$.

Having found $|X|$ such copies of $T$ in this way, deleting all these copies yields subsets $A' \subseteq A$ and $B' \subseteq B$ with $|A'| = |A| - 4|X| = 3n/5 - 3|X| = 3n_1/5$ and $|B'| = |B| - |X| = 2n/5 - 2|X| = 2n_1/5$, where $n_1 := n-5|X| = |A' \cup B'|$ (so in particular $n-50\sqrt{\gamma} n \leq n_1 \leq n$ and $5 \mid n_1$). Crucially the ratio of sizes of $A'$ and $B'$ is $|A'| : |B'| = 3 : 2$. Let $H' := H[A' \cup B']$, so $H'$ has $n_1$ vertices. We then have the following properties.
\begin{enumerate}[(i)]
\item At most $\gamma' n_1^2$ pairs of vertices in $A'$ are not good. Indeed, each such pair is either one of the at most $\sqrt{\gamma} n^2$ bad pairs or includes one of the at most $10\sqrt{\gamma} n$ vertices of $X$.\label{i0}
    \item Every good pair has at least $2n_1/5-\gamma' n_1$ neighbours in $B'$. This follows from the fact that each good pair had at least $2n/5 - 11\sqrt{\gamma} n$ neighbours in $B$. 
    \label{i1}
    \item $\delta(H') \geq 2n/5 - 5|X| \geq 2n_1/5 - \gamma' n_1$, since $5|X|$ vertices were deleting in forming $H'$. \label{i2}
    \item Every vertex in $B'$ is in at least $n_1^2/100$ edges with precisely two vertices in $A'$. Indeed, this follows from the definition of $X$ since each vertex in $B'$ is not in $X$ and so so formed an edge of $H$ with at least $n^2/50$ pairs in $A$. \label{i3}
\end{enumerate}
Let $J$ be the 5-graph on $V(H')$ whose edges are all sets $S \subseteq V(H')$ with $|S \cap A'| = 3$ and $|S \cap B'| = 2$ which support a copy $T_S$ of $T$ in $H'$. We will show that $J$ has the following properties. 
\begin{enumerate}[(P1)]
\item $e(J) \geq (1-\beta)\binom{3n_1/5}{3}\binom{2n_1/5}{2}$. \label{item_p1}
\item Every vertex of $H'$ is contained in at least $\frac{1}{10^{8}}(\frac{n_1}{5})^4$ edges of $J$. \label{item_p2}
\end{enumerate}
Using (P\ref{item_p1}) and (P\ref{item_p2}) we may apply Corollary~\ref{auxiliarymatching} to obtain a perfect matching in $J$. The copies $T_S$ of $T$ for each edge $S$ in this matching then form a perfect $T$-tiling in $H'$. Together with the $|X|$ copies of $T$ previously set aside in $H$, this yields a perfect $T$-tiling in $H$, completing the proof.

It remains to verify the properties (P\ref{item_p1}) and (P\ref{item_p2}). For (P\ref{item_p1}) observe that by (i) there are at most~$\gamma' n_1^3$ triples $\{x_1, x_2, x_3\}$ with $x_1, x_2, x_3 \in A'$ for which not all of $x_1x_2$, $x_2x_3$, $x_3x_1$ are good pairs. On the other hand, if $x_1x_2, x_2x_3, x_1x_3 \in \binom{A'}{2}$ are each good pairs, then by (ii) we have $|B' \setminus (N_H(x_1x_2) \cap N_H(x_2x_3) \cap N_H(x_1x_3))| \leq 3\gamma' n_1$. It follows that the number of 5-tuples $S = \{x_1, x_2, x_3, y_1, y_2\}$ with $x_1, x_2, x_3 \in A'$ and $y_1, y_2 \in B'$ for which we do not have $y_1, y_2 \in N_H(x_1x_2) \cap N_H(x_2x_3) \cap N_H(x_1x_3)$ is at most $\gamma' n_1^3 \cdot \binom{|B'|}{2} + \binom{|A'|}{3} \cdot 3 \gamma' n_1 \cdot |B'| \leq 4 \gamma' n_1^5$. For every other 5-tuple $S = \{x_1, x_2, x_3, y_1, y_2\}$ with $x_1, x_2, x_3 \in A'$ and $y_1, y_2 \in B'$ we have $S \in E(J)$ since $S$ supports a copy of $T$ with edges $x_1x_2y_1$, $x_1x_3y_1$ and $x_2x_3y_2$. We conclude that $e(J) \geq \binom{3n_1/5}{3}\binom{2n_1/5}{2} - 4 \gamma' n_1^5$, and~(P\ref{item_p1}) follows.

For (P\ref{item_p2}), first consider an arbitrary vertex $b \in B'$. By~(\ref{i3}) there are at least $n_1/100$ vertices $a \in A'$ with $|N(ab) \cap A'| \geq n_1/100$. For each such $a$ we can form a copy of $T$ in $H'$ which contains $b$ and precisely three vertices in $A'$ by next choosing a good pair $a_1, a_2 \in N(ab) \cap A'$, then choosing $b_1 \in N(a_1a_2) \cap B'$ to obtain a copy of $T$ with edges $aba_1, aba_2$ and $a_1a_2b_1$. By~(\ref{i0}) there are at least $\binom{n_1/100}{2} - \gamma' n_1^2 \geq n_1^2/10^5$ options for the good pair $\{a_1, a_2\}$, and because this pair is good, by~(\ref{i1}) there are at least $2n_1/5 - \gamma'n_1 \geq n_1/10$ options for $b_1$. So in total we obtain at least $n_1^4/10^8$ copies of $T$ in $H'$ which contain $b$ and have precisely three vertices in $A'$. Since each set of five vertices supports at most $5! \leq 5^4$ copies of $T$, it follows that (P\ref{item_p2}) holds for $b$.

Now fix $a \in A'$. Suppose first that for more than $n_1/100$ vertices $b \in B'$ we have $|N(ab) \cap A'| \geq n_1/100$. For each such $b$ we can form copies of $T$ in $H'$ which contain $a$ and precisely two vertices of $B'$ by exactly the same argument as we used for each $a$ in the previous paragraph. Just as there, it follows that (P\ref{item_p2}) holds for $a$. 

Next suppose instead that there are more than $n_1/5$ vertices $a' \in A'$ with $|N(aa') \cap A'| \leq n_1/10$. Let $A_1$ be the set of these vertices $a'$, so $|A_1| > n_1/5$. We may form a copy of $T$ in $H'$ containing $a$ and precisely two vertices of $B'$ as follows. First choose a good pair $a_1a_2 \in \binom{A_1}{2}$, then choose $b_1 \in N(aa_1) \cap N(aa_2) \cap B'$, then choose $b_2 \in N(a_1a_2) \cap B'$. This gives the desired copy of $T$ with edges $aa_1b_1, aa_2b_1$ and $a_1a_2b_2$. By~(\ref{i0}) there are at least $\binom{n_1/5}{2} - \gamma' n_1^2 \geq n_1^2/100$ options for the good pair $a_1a_2$, and by definition of $A_1$ and~(\ref{i2}) there are then at least $2n_1/5 - 2\cdot (n_1/10) - 2 \gamma' n_1 \geq n_1/10$ options for $b_1$. Finally, since $a_1a_2$ is good, by~(\ref{i1}) there are at least $n_1/5$ options for $b_2$, so in total we obtain at least $n_1^4/10^4$ copies of $T$ containing $a$ and precisely two vertices of $B$. Similarly as before it follows that (P\ref{item_p2}) holds for $a$.

The final possibility is that there are at least $2n_1/5$ vertices $a' \in A'$ with $|N(aa') \cap A'| > n_1/10$ and at least $2n_1/5-n_1/100 \geq n_1/5$ vertices $b \in B'$ with $|N(ab) \cap A'| < n_1/100$. Observe that by~(\ref{i2}) each such $b$ has $|N(ab) \cap B'| \geq n_1/5$. Choose any such $b \in B'$, and then choose $a',a_1 \in A'$ such that $a'a_1$ is a good pair and $aa'a_1 \in E(H')$, before finally choosing $b_1 \in N(ab) \cap N(a'a_1) \cap B'$. This gives a copy of $T$ in $H'$ with edges $aa'a_1$, $abb_1$, $a'a_1b_1$ which contains $a$ and has precisely two vertices in $B'$. There are at least $n_1/5$ options for $b$ and, by~(\ref{i0}), at least $(1/2) \cdot (2n_1/5) \cdot (n_1/10) - \gamma'n_1^2 \geq n_1^2/100$ options for the good pair $a'a_1$. Since $a'a_1$ is good,~(\ref{i1}) and our choice of $b$ imply that there are at least $n_1/10$ options for $b_1$. So in total we obtain at least $n_1/10^4$ copies of $T$ containing $a$ and precisely two vertices of $B$. Similarly as before it follows that (P\ref{item_p2}) holds for $a$, completing the proof.
\end{proof}

\section{Absorbing} \label{sec:absorbing}
In this section we prove the absorbing lemma (Lemma~\ref{lem:absorbing}). For this we follow the approach of Lo and Markstr\"om~\cite{LM}, beginning with the following definitions. Let~$H$ be a~$3$-graph on~$n$ vertices. Say that vertices~$u, v \in V(H)$ are \emph{$(\eta, r)$-linked in~$H$} if there exist at least~$\eta \binom{n}{5r-1}$ sets~$S \in \binom{V(H)}{5r-1}$ such that both~$H[S \cup \{u\}]$ and~$H[S \cup \{v\}]$ admit perfect~$T$-tilings. A set~$X \subseteq V(H)$ is \emph{$(\eta, r)$-closed in~$H$} if for every~$x, y \in X$ we have that~$x$ is~$(\eta, r)$-linked to~$y$ in $H$. We will obtain our absorbing set by applying the following special case of a theorem of Lo and Markstr\"om (the full version of their result gives an analogous statement for all~$k$-uniform hypergraphs, but we only state the result as it applies to the generalised triangle 3-graph which we consider in this manuscript).

\begin{theorem}[\cite{LM}, Lemma~1.1] \label{thm:lomarkstrom}
For each positive integer~$t$ and each~$\eta > 0$, there exists~$n_0$ for which the following statement holds. If~$H$ is a~$3$-graph on~$n \geq n_0$ vertices such that~$V(H)$ is~$(\eta, t)$-closed in~$H$, then there exists a set~$A \subseteq V(H)$ of size~$|A| \leq \eta^5 n/2560t$ for which there exists a perfect~$T$-tiling in~$H[A \cup S]$ for every set~$S \subseteq V \setminus A$ of size~$|S| \leq \eta^{10} n/ (10 \cdot 2^{18} \cdot t^2)$ with~$5 \mid |S|$.\end{theorem}

To enable us to apply Theorem~\ref{thm:lomarkstrom}, we will show that the minimum degree condition of Lemma~\ref{lem:absorbing} on a $3$-graph $H$ is sufficient to ensure that~$V(H)$ is~$(\eta, t)$-closed in $H$ for appropriate~$\eta$ and~$t$. This is the following lemma, which the remaining part of this section is devoted to proving.

\begin{lemma} \label{lem:linked}
Suppose that~$1/n \ll \eta \ll 1/t \ll \alpha$, where~$t$ is an integer. If~$H$ is a 3-graph on~$n$ vertices with~$\delta(H) \geq n/3+\alpha n$, then~$V(H)$ is~$(\eta, t)$-closed in~$H$.
\end{lemma}
Combining Theorem~\ref{thm:lomarkstrom} and Lemma~\ref{lem:linked} immediately gives a proof of Lemma~\ref{lem:absorbing}. 

\begin{proof}[Proof of Lemma~\ref{lem:absorbing}]
    Given constants~$\alpha$ and~$c$ and a~$3$-graph~$H$ on~$n$ vertices with~$\delta(H) \geq n/3 + \alpha n$ as in the statement of Lemma~\ref{lem:absorbing}, introduce a new integer~$t$ with~$c \ll 1/t \ll \alpha$ and set $\eta := \sqrt[5]{2560tc}$, so~$c = \eta^5/2560t$. We then have $1/n \ll \eta \ll 1/t \ll \alpha$, so   Lemma~\ref{lem:linked} implies that~$V(H)$ is~$(\eta, t)$-closed in~$H$. Theorem~\ref{thm:lomarkstrom} then yields a set~$A \subseteq V(H)$ of size~$|A| \leq cn$ with the desired absorbing property (to see this, observe that our choice of~$\eta$ ensures that~$c^2 \leq \eta^{10}/ (10 \cdot 2^{18} \cdot t^2)$).
\end{proof}

It remains to prove Lemma~\ref{lem:linked}. We use the following theorem of Han and Treglown~\cite{HT}, which gives conditions to ensure that the vertex set of a~$3$-graph~$H$ may be partitioned into not-too-small parts, each of which is closed in~$H$. Note that whilst our definitions of linkedness and closure are specific to the generalised triangle $3$-graph~$T$ considered in this manuscript, the full version of their theorem gives an analogous statement for all~$k$-uniform hypergraphs.

\begin{theorem}[\cite{HT}, Lemma 6.3] \label{thm:HT}
Suppose that~$1/n \ll \eta' \ll \eta, 1/c, \beta$, where~$c$ is an integer. If~$H$ is a~$3$-graph on~$n$ vertices in which 
\begin{enumerate}[(i)]
    \item for each~$u \in V(H)$ there are at least~$\beta n$ vertices~$v \in V(H)$ for which~$u$ and~$v$ are~$(\eta, 1)$-linked in~$H$, and 
    \item every set of~$c+1$ vertices of~$H$ contains two distinct vertices which are~$(\eta, 1)$-linked in~$H$,
\end{enumerate}
then there exists a partition~$\Part$ of~$V(H)$ into~$V_1, \dots, V_r$ with~$r \leq \min(c, 1/\beta)$ such that for every~$i \in [r]$ we have~$|V_i| \geq (\beta - \eta)n$ and~$V_i$ is~$(\eta', 2^{c-1})$-closed in~$H$.
\end{theorem}

Indeed, the next lemma shows that the minimum codegree condition~$\delta(H) \geq n/3+\alpha n$ on a 3-graph~$H$ ensures that the conditions of Theorem~\ref{thm:HT} are met for~$\beta = \alpha/2$ and~$c=2$.

\begin{lemma} \label{lem:15parts}
Suppose that~$1/n \ll \eta \ll \alpha$. If~$H$ is a 3-graph on~$n$ vertices with~$\delta(H) \geq n/3+\alpha n$, then 
\begin{enumerate}[(i)]
    \item every vertex of~$V(H)$ is~$(\eta, 1)$-linked to at least~$\alpha n/2$ vertices of~$V(H)$, and 
    \item every set of three vertices of~$H$ contains two distinct vertices which are~$(\eta, 1)$-linked in~$H$.
\end{enumerate}
\end{lemma}
\begin{proof}
For (i), fix~$z \in V(H)$, and for each~$u \in V(H) \setminus \{z\}$ set~$W(u) := N(uz)$.
Observe that for each triple of distinct vertices~$w_1,w_2,w_3 \in W(u)$ there exist $i, j \in [3]$ with $i \neq j$ for which~$|N(uw_i) \cap N(uw_j)| \geq \alpha n$ (this follows from the minimum codegree condition by averaging). Since each pair of vertices is contained in at most $n$ triples, and $|W(u)| \geq n/3$, it follows that there are at least~$\binom{|W(u)|}{3}/n \geq n^2/175$ distinct pairs~$w, w' \in W(u)$ with $w \neq w'$ for which~$|N(uw)\cap N(uw')| \geq \alpha n$.
For each such $u, w$ and $w'$, and for each~$z' \in N(uw)\cap N(uw')$ and~$v \in N(ww') \setminus \{z, z', u\}$, we have that~$\{z,u,w,w',v\}$ and~$\{z',u,w,w',v\}$ both support copies of~$T$.

Call a vertex $z'$ \emph{good} if there are at least $\eta n^4$ sets $S$ for which both $S \cup \{z\}$ and $S \cup \{z'\}$ support copies of $T$, 
and call $z'$ \emph{bad} otherwise. Suppose for a contradiction that there are fewer than $\alpha n/2$ good~$z'$. Let~$X$ be the collection of pairs $(S,z')$ in which $z'$ is bad and $S \cup \{z\}$ and $S \cup \{z'\}$ both support copies of~$T$. On the one hand, picking $u,w,w',v$ as above yields at least~$(n-1) \cdot (n^2/175) \cdot (n/3) \geq n^4/600$ choices of~$S$ for which~$S \cup \{z\}$ supports a copy of~$T$ and for which there are at least $\alpha n$ vertices $z'$ such that $S \cup \{z'\}$ supports a copy of $T$. At least $\alpha n/2$ of these vertices must be bad, so we have $|X| \ge \alpha n^5/1200$. On the other hand, there are at most~$n$ bad vertices~$z'$, which by definition each have at most~$\eta n^4$ sets~$S$ with $(S,z') \in X$, so~$|X| \leq \eta n^5$. This implies that~$\eta \geq \alpha/1200$, contradicting our choice of $\eta$ and $\alpha$. So there must be at least $\alpha n/2$ good vertices, meaning that~$z$ is~$(\eta, 1)$-linked to at least~$\alpha n/2$ vertices. This completes the proof of (i).

For (ii), let $w_1,w_2,w_3$ be three distinct vertices of $H$. As above, for each $u \notin \{w_1,w_2,w_3\}$ we have $|N(uw_i)\cap N(uw_j)| \ge \alpha n$ for some $i, j \in \{1,2,3\}$ with $i \neq j$. So, without loss of generality, we may assume there is a set $U \subseteq V(H)\setminus \{w_1,w_2,w_3\}$ of cardinality at least $n/4$ such that $|N(uw_1)\cap N(uw_2)| \ge \alpha n$ for every $u \in U$. Now observe that for each $u \in U$, each $y,y' \in N(uw_1)\cap N(uw_2)$, and each $x \in N(yy') \setminus \{u, w_1, w_2\}$ both $\{w_1,u,y,y',x\}$ and $\{w_2,u,y,y',x\}$ support copies of $T$. Noting that there are at least $n/4$ choices for $u$, at least $\alpha n/2$ choices for each of $y$ and $y'$ and at least $n/3$ choices for $x$ gives (ii).
\end{proof}

We also use the following theorem of Han, Zang and Zhao~\cite{HZZ}, which gives a condition on a partition of~$V(H)$ into closed parts which ensures that actually~$V(H)$ itself is closed. As before, whilst our definition of closure only applies to the generalised triangle~$3$-graph~$T$ considered in this manuscript, the full version of the theorem gives an analogous statement for all~$k$-uniform hypergraphs. The theorem makes use of the following definitions. 
 Let~$\mathcal{P} = (V_1, V_2, \dots, V_r)$ be a partition of~$V(H)$ equipped with an order on its parts. For a set~$S \subseteq V(H)$ the index vector~$i(S)$ of~$S$ (with respect to~$\mathcal{P}$) is the vector~$(|S \cap V_1|, \dots, |S \cap V_r|) \in \mathbb{Z}^r$. Write~$I_\mathcal{P}(H)$ for the set of all vectors~$i \in \mathbb{Z}^r$ whose entries are all non-negative and sum to five; these are the possible index vectors for copies of $T$ in $H$. Define~$I^\mu_\mathcal{P}(H)$ to be the set of all $i \in I_\mathcal{P}(H)$ for which there exist at least~$\mu \binom{n}{5}$ copies of~$T$ in~$H$ with index vector~$i$ (with respect to~$\mathcal{P}$), and let $L^\mu_\mathcal{P}(H)$ be the lattice generated by $I^\mu_\mathcal{P}(H)$, that is, the additive subgroup of $\mathbb{Z}^r$ which contains $I^\mu_\Part(H)$ and which is minimal with respect to inclusion among all such subgroups. We write~$u_j$ to denote the unit vector from~$\mathbb{Z}^r$ with a~$1$ in position~$j$.

\begin{theorem}[\cite{HZZ}, Lemma 3.9] \label{thm:HZZ}
     For all~$\eta, \mu, \eps$ and~$t$ there exist~$\eta', t'$ and~$n_0$ for which the following holds. Let~$H$ be a 3-graph on~$n \geq n_0$ vertices, and let~$\Part$ be a partition of~$V(H)$ into parts~$V_1, \dots, V_r$ where for each~$j \in [r]$ the part~$V_j$ has size~$|V_j| \geq \eps n$ and is~$(\eta, t)$-closed in~$H$. If~$u_\ell - u_{\ell'} \in L_\Part^{\mu}$ for every~$\ell, \ell' \in [r]$, then~$V(H)$ is~$(\eta', t')$-closed in~$H$.  
 \end{theorem}

To apply Theorem~\ref{thm:HZZ}, we show that for every partition~$\Part$ of~$V(H)$ into two not-too-small parts, there are index vectors differing only by a difference of unit vectors which are both well-represented by copies of~$T$. 

\begin{lemma} \label{lem:2parts}
Suppose that~$1/n \ll \psi \ll \alpha, \beta$, and let~$H$ be a 3-graph on~$n$ vertices with~$\delta(H) \geq n/3 + \alpha n$. If~$\mathcal{P}=(A,B)$ is a vertex partition of~$V(H)$ such that~$|A|, |B| \geq \beta n$, then there exist index vectors~$i, i' \in I_\mathcal{P}(H)$ such that~$i - i' = (1,-1)$ and~$H$ contains at least~$\psi n^5$ copies of~$T$ with~$i(T) = i$ and at least~$\psi n^5$ copies of~$T$ with~$i(T) = i'$ (so in particular $i, i' \in I^\psi_\mathcal{P}(H)).$
\end{lemma}

Note that index vectors of copies of $T$ with respect to~$(A, B)$ have the form~$(x, 5-x)$ for some~$x$, so we want, for example,~$\psi n^5$ copies with index (2,3) and~$\psi n^5$ copies with index vector (3,2).

\begin{proof}
    Introduce a new constant $\gamma$ with $\psi \ll \gamma \ll \alpha, \beta$. For each pair of vertices $u, v \in V(H)$, say that $\{u,v\}$ is an \emph{aligned} pair if $u, v \in A$ or $u, v \in B$, and that $\{u, v\}$ is a \emph{split} pair if $|\{u,v\} \cap A| = 1$. We consider two cases.
    
\medskip
\noindent {\em Case 1: No $S \in \{A,B\}$ has $|N(xy) \cap S| \geq \alpha n$ for all but at most $\gamma n^2$ aligned pairs $\{x,y\}$.}
\medskip

   In this case, let $M_A$ be the set of aligned pairs with at least $n/3$ neighbours in $A$ and let $M_B$ be the set of aligned pairs with at least $n/3$ neighbours in $B$. Our case assumption implies that there are at least $\gamma n^2$ aligned pairs that have at most $\alpha n$ neighbours in $A$, and there are at least $\gamma n^2$ aligned pairs that have at most $\alpha n$ neighbours in $B$. Since $\delta(H) \geq n/3+\alpha n$, this immediately gives $\size{M_A}, \size{M_B} \ge \gamma n^2$ and that $\size{A}, \size{B} \geq n/3$. 

Now let $\{u, v\}$ be a split pair. If $|N(uv) \cap A| \geq |A| - n/3 + \alpha n/2$, then at most $n/3 - \alpha n/2$ vertices of $A$ are not in $N(uv)$, so for every $\{x_1, y_1\} \in M_A$ we have $|N(uv) \cap N(x_1y_1) \cap A| \geq 
\alpha n/2$. In this case we say that $\{u, v\}$ is \emph{good for $M_A$}. If instead $|N(uv) \cap A| < |A| - n/3 + \alpha n/2$ then we have
$$|N(uv) \cap B| \geq \delta(H) - |N(uv) \cap A| > \left(\frac{n}{3} + \alpha n\right) - \left(|A| - \frac{n}{3} + \frac{\alpha n}{2}\right) = |B| - \frac{n}{3} + \frac{\alpha n}{2},$$
so for every $\{x_2, y_2\} \in M_B$ we have $|N(uv) \cap N(x_2y_2) \cap B| \geq \alpha n/2$. In this case we say that $\{u, v\}$ is \emph{good for $M_B$}. So every split pair $\{u, v\}$ is either good for $M_A$ or good for $M_B$. 

        For each split pair $\{u,v\}$, let $J(u,v)\in \{A,B\}$ be such that $\{u,v\}$ is good for $M_J$. 
Let $\mathcal{S}$ denote the set of all sequences $S = (u, v, x, y, z_1, z_2, w)$ of distinct vertices of $H$ such that (1) $\{u,v\}$ is a split pair; (2) $\{x,y\} \in M_{J(u, v)}$; (3) $z_1,z_2 \in N(uv) \cap N(xy) \cap J(u, v)$; and (4) $w \in N(z_1z_2)$. Choosing the vertices of $S$ in turn, there are at least $n^2/9$ options for $(u, v)$, at least $\gamma n^2$ options for $(x,y)$, at least $\binom{\alpha n/2}{2} \geq \alpha^2 n^2/9$ options for $(z_1,z_2)$, and at least $n/3$ options for $w$. So $|\mathcal{S}| \geq \gamma \alpha^2 n^7/243 \geq 40 \psi n^7$.

Observe that for each $S \in \mathcal{S}$ we obtain copies $T_1(S)$ and $T_2(S)$ of $T$ in $H$, where $T_1(S)$ has edges $uvz_1, uvz_2, z_1z_2w$ and $T_2(S)$ has edges $xyz_1, xyz_2$ and $z_1z_2w$. Moreover, since $\{u,v\}$ is split and $\{x,y\}$ is aligned we have $i(T_1(S)) - i(T_2(S)) \in \{(1,-1), (-1, 1)\}$. There are at most ten possibilities for such a pair $(i(T_1(S)), i(T_2(S)))$, so by averaging there exist $i$ and $i'$ with $i - i' \in \{(1,-1), (-1, 1)\}$ for which at least $4\psi n^7$ sequences $S \in \mathcal{S}$ have $i(T_1(S)) = i$ and $i(T_2(S)) = i'$. Fix such $i$ and $i'$. For each copy $T'$ of $T$ in $H$ there are at most $4n^2$ sequences $S \in \mathcal{S}$ with $T_1(S) = T'$ and at most $4n^2$ sequences $S \in \mathcal{S}$ with $T_2(S) = T'$. It follows that at least $\psi n^5$ copies of $T$ in $H$ have index vector $i$, and at least $\psi n^5$ copies of $T$ in $H$ have index vector $i'$, as required.

\medskip \noindent {\em Case 2: There is $S \in \{A, B\}$ such that $|N(xy) \cap S| \geq \alpha n$ for all but at most $\gamma n^2$ aligned pairs $\{x,y\}$.}

\medskip
    Assume without loss of generality that $S = A$. Say that an aligned pair $\{u,v\}$ is \emph{good} if $|N(uv) \cap A| \geq \alpha n$. So our case assumption is that at most $\gamma n^2$ pairs are not good; in particular, we must have $|A| \geq \alpha n$. 
    We may construct a copy of $T$ with index vector $(5,0)$ in the following way. First choose a good pair $\{u, v\}$ with $u, v \in A$. Next choose a good pair $z_1, z_2 \in N(uv) \cap A$, and finally choose $z_3 \in N(z_1 z_2) \cap (A \setminus \{u, v\})$. This gives a copy of $T$ in $H[A]$ with edges $uvz_1, uvz_2, z_1z_2z_3$. When making these choices, there are at least $\binom{\alpha n}{2} - \gamma n^2 \geq \alpha^2 n^2/3$ options for $\{u, v\}$. Then (since $\{u, v\}$ is good) there are at least $\binom{\alpha n}{2}  - \gamma n^2 \geq \alpha^2 n^2/3$ options for $\{z_1, z_2\}$, and finally at least $\alpha n - 2$ options for $z_3$. So in total we obtain at least $(\alpha^2 n^2/3) \cdot (\alpha^2 n^2 /3) \cdot (\alpha n - 2) \geq \psi n^5$ copies of $T$ with index vector $(5,0)$.

Now suppose that there are more than $\gamma n^2$ split pairs $\{u, v\}$ with $|N(u v) \cap B| < n/3$. Each such pair can be extended to a copy of $T$ in $H$ with index vector (4, 1) in the following way. First choose a good pair $\{z_1, z_2\}$ with $z_1, z_2 \in N(uv) \cap A$, and then choose $z_3 \in N(z_1z_2) \cap (A \setminus \{u, v\})$. This gives a copy of $T$ with edges $uvz_1, uvz_2$ and $z_1z_2z_3$ which has index vector (4, 1). Since $|N(uv) \cap A| > \alpha n$, when making these choices there are at least $\binom{\alpha n}{2} - \gamma n^2 \geq \alpha^2 n^2/3$ options for $\{z_1, z_2\}$, and then (since $\{z_1, z_2\}$ is good) at least $\alpha n-1 $ options for $z_3$. In total this gives rise to at least $\gamma n^2 \cdot (\alpha^2 n^2/3) \cdot (\alpha n-1) \geq \psi n^5$ copies of $T$ with index vector (4,1), so taking $i= (5, 0)$ and $i' = (4, 1)$ gives the desired outcome.

We therefore assume that all but at most $\gamma n^2$ split pairs $\{u, v\}$ have $|N(u v) \cap B| \geq n/3$. Since there are $|A||B| \geq \beta n \cdot (1-\beta) n \geq \beta n^2/2 $ split pairs in total, it follows that there are at least $\beta n^2/2 - \gamma n^2 \geq \beta n^2/3$ split pairs $\{u, v\}$ with $|N(u v) \cap B| \geq n/3$. Call these pairs {\em useful}, and observe that their existence implies that $|B| \geq n/3$.

Each useful pair may be extended to a copy of $T$ in $H$ with index vector (2, 3) in the following way. First choose a good pair $\{z_1, z_2\}$ with $z_1, z_2 \in N(uv) \cap B$, then choose $z_3 \in N(z_1z_2) \cap (A \setminus \{u, v\})$. This gives a copy of $T$ with edges $uv z_1, uvz_2$ and $z_1z_2z_3$ which has index vector (2, 3). When making these choices, since $\{u, v\}$ is useful and at most $\gamma n^2$ pairs are not good, there are at least $\binom{n/3}{2} - \gamma n^2 \geq n^2/20$ options for the pair $\{z_1, z_2\}$, and since $\{z_1, z_2\}$ is good there are at least $\alpha n-1$ options for $z_3$. So in total we obtain at least $(\beta n^2/3) \cdot (n^2/20) \cdot (\alpha n - 1) \geq \psi n^5$ copies of $T$ with index vector (2, 3). 

In a similar way, each good pair $\{u, v\} $ with $u, v \in B$ may be extended to a copy of $T$ with index vector (3, 2) in the following way. First choose a good pair $\{z_1, z_2\}$ with $z_1, z_2 \in N(uv) \cap A$, and then choose $z_3 \in N(z_1z_2) \cap A$. This gives a copy of $T$ in $H$ with edges $uv z_1, uvz_2$ and $z_1z_2z_3$ which has index vector (3, 2). When making these choices, since $\{u, v\}$ is good and at most $\gamma n^2$ pairs are not good, there are at least $\binom{\alpha n}{2} - \gamma n^2 \geq \alpha^2 n^2/3$ options for the pair $\{z_1, z_2\}$. Since $\{z_1, z_2\}$ is good there are then at least $ \alpha n$ options for $z_3$. Since there are at least $\binom{|B|}{2} - \gamma n^2 \geq n^2/20$ good pairs $\{u, v\}$ with $u, v \in B$, overall we obtain at least $(n^2/20) \cdot (\alpha^2n^2/3) \cdot \alpha n > \psi n^5$ copies of $T$ in $H$ with index vector (3, 2). So we may take $i = (3, 2)$ and $i' = (2,3)$.
    \end{proof}

Finally, by combining the previous results we prove Lemma~\ref{lem:linked}.

\begin{proof}[Proof of Lemma~\ref{lem:linked}]
Introduce new constants $\mu, \eta'$ and $\eta''$ with~$1/n \ll \eta \ll 1/t \ll \mu, \eta' \ll \eta'' \ll \alpha$ and let~$H$ be a~$3$-graph on~$n$ vertices with~$\delta(H) \geq n/3+\alpha n$. By Lemma~\ref{lem:15parts} every vertex~$u \in V(H)$ is~$(\eta'', 1)$-linked in~$H$ to at least~$\alpha n/2$ other vertices of~$V(H)$, and every set of three distinct vertices of~$H$ contains two distinct vertices which are~$(\eta'', 1)$-linked in~$H$. So we may apply Theorem~\ref{thm:HT} with $c=2$. This tells us that either there is a partition~$\Part$ of~$V(H)$ into two parts~$U_1$ and $U_2$, with $|U_1|, |U_2| \geq (\alpha/2 - \eta'')n$, for which $U_1$ and $U_2$ are each~$(\eta', 2)$-closed in~$H$, or $V(H)$ is itself $(\eta', 2)$-closed in $H$. We may assume the former, since in the latter case we may take an arbitrary near-balanced partition of $V(H)$ into parts $U_1$ and $U_2$. By Lemma~\ref{lem:2parts} there exist index vectors~$i, i' \in I^\mu_{\Part}(H)$ for which~$i - i' = u_{1}-u_{2}$. So for each $\ell, \ell' \in \{1, 2\}$, if $\ell \neq \ell'$ then $u_\ell -u_{\ell'} \in L_\Part^\mu(H)$ since $I^\mu_\Part(H)\subseteq L^\mu_\Part(H)$, whilst if $\ell = \ell'$ then $u_\ell -u_{\ell'} \in L_\Part^\mu(H)$ since $0 \in L^\mu_\Part(H)$. So we may apply Theorem~\ref{thm:HZZ} to conclude that~$V(H)$ is~$(\eta, t)$-closed in~$H$, as required.
\end{proof}

\section{Almost-perfect tilings} \label{sec:tiling}
The aim of this section is to prove the tiling lemma (Lemma~\ref{lem:tiling}). We do this by first proving a fractional version of the result (Lemma~\ref{frac_tiling_with_bound}), from which Lemma~\ref{lem:tiling} follows by applying a result of Pippenger and Spencer~\cite{PS}. Towards this end we make the following definitions regarding fractional tilings in $k$-graphs, all of which are standard within this theory.

Let $H$ and $F$ be $k$-graphs, and let $\mathcal{F}(H)$ denote the set of all copies $F'$ of $F$ in $H$. Also for each $u, v \in V(H)$ let $\mathcal{F}_u(H)$ denote the set of all $F' \in \mathcal{F}(H)$ with $u \in V(F')$ and $\mathcal{F}_{uv}(H)$ denote the set of all $F' \in \mathcal{F}(H)$ with $u,v \in V(F')$. 
A \emph{fractional $F$-tiling} in $H$ is a function $w : \mathcal{F}(H) \to [0,1]$ such that every $u \in V(H)$ has $\sum_{F' \in \mathcal{F}_u(H)} w(F') \leq 1$. 
This is the linear relaxation of the (non-fractional) tiling problem: if $w(F') \in \{0,1\}$ for every $F' \in \mathcal{F}(H)$ then $\{F' \in \mathcal{F}(H) : w(F') = 1\}$ is an $F$-tiling in $H$, that is, a set of pairwise vertex-disjoint copies of $F$ in $H$. 
Say that $w$ is \emph{perfect} if every $u \in V(H)$ has $\sum_{F' \in \mathcal{F}_u(H)} w(F')= 1$; observe that this is equivalent to saying that $\sum_{F' \in \mathcal{F}(H)} w(F') = n/|V(F)|$. 
For vertices $u, v \in V(H)$ we write $w(u) = \sum_{F' \in \mathcal{F}_u(H)} w(F')$ and $w(uv) = \sum_{F' \in \mathcal{F}_{uv}(H)} w(F')$. Now let $B$ be a graph with $V(B) = V(H)$. 
Say that a copy $F'$ of $F$ in $H$ is \emph{$B$-avoiding} if there is no edge $uv \in E(B)$ with $u, v \in V(F')$. 
Similarly, we say that a fractional $F$-tiling $w$ in $H$ is \emph{$B$-avoiding} if every copy $F'$ of $F$ with $w(F') > 0$ is $B$-avoiding; note that this is equivalent to saying that $w(uv) = 0$ for every edge $uv \in E(B)$.

When we consider $F$-tilings in the specific case where $F = T$ is the generalised triangle, we write $\mathcal{T}(H), \mathcal{T}_u(H)$ and $\mathcal{T}_{uv}(H)$ in place of $\mathcal{F}(H), \mathcal{F}_u(H)$ and $\mathcal{F}_{uv}(H)$.

Now suppose that $H$ is a $k$-graph with $n$ vertices, and identify the vertex set of $H$ with $[n]$. For each set $U \subseteq V(H)$, we define the \emph{characteristic vector of $U$}, denoted $\mathbbm{1}_U$, to be the vector in $\mathbb{R}^n$ with
$$
(\mathbbm{1}_U)_i=
\begin{cases}
1 \mbox{~if~} i \in U,\\
0 \mbox{~otherwise.}
\end{cases}
$$
The {\it positive cone} of a set of $t$ vectors $\mathbf{v}_1, \ldots, \mathbf{v}_t \in \mathbb{R}^n$ is the set 
$$\pc(\{\mathbf{v}_1, \ldots, \mathbf{v}_t\}) = \left\{\sum_{i \in [t]} \lambda_i \mathbf{v}_i : 
\mbox{$\lambda_i \in \mathbb{R}, \lambda_i \geq 0$ for each $i \in [t]$}
\right\}.$$

Let $\mathbf{1} \in \mathbb{R}^n$ denote the all-ones vector, and observe that $H$ has a perfect fractional $F$-tiling if and only if $\mathbf{1} \in \pc(\{\mathbbm{1}_{V(F')} : F' \in \mathcal{F}(H)\})$ (given by setting $w(F')$ to be the corresponding coefficient $\lambda_i$ for each $F' \in \mathcal{F}(H)$). 
In this context, Farkas' lemma on solvability of systems of linear inequalities gives a useful consequence of the non-existence of a perfect fractional $F$-tiling in $H$, 
namely that there exists ${\bf a} \in \mathbb{R}^n$
for which $\mathbf{a} \cdot \mathbbm{1}_{V(F')} \geq 0$ 
for every $F' \in \mathcal{F}(H)$ and 
${\bf a} \cdot {\bf 1} <0$.

\begin{lemma}[Farkas' lemma] \label{lem:farkas}
    If $X \subseteq \mathbb{R}^n$ is finite and ${\bf y} \in \mathbb{R}^n \setminus \pc(X)$, then there exists some ${\bf a} \in \mathbb{R}^n$ such that ${\bf a}\cdot {\bf x} \geq 0$ for all ${\bf x} \in X$ and ${\bf a} \cdot {\bf y} <0$.
\end{lemma}

Finally, let $U =\{u_1, u_2, \ldots, u_{\ell}\}$ and $W = \{w_1, w_2, \ldots, w_{\ell}\}$ each be sets of $\ell$ vertices of $H$, ordered with $u_1 < u_2 < \dots < u_\ell$ and $w_1 < w_2 < \dots < w_\ell$. We say that $U$ {\it dominates} $W$ if $w_i \leq u_i$ for each $i \in [\ell]$.

Our next lemma gives a fractional version of the tiling lemma, showing that the same minimum codegree condition ensures that a $3$-graph on $n$ vertices either is close to extremal or contains a perfect fractional matching. Moreover, we may insist that the edges of non-zero weight in this fractional matching avoid a given small set $B$ of pairs of vertices.

\begin{lemma} \label{frac_tiling} 
Suppose that $1/n \ll \eps, \alpha \ll \gamma$, and let $V$ be a set of $n$ vertices. If $H$ is a $3$-graph on $V$ with
$\delta(H) \geq 2n/5 - \alpha n$ and $B$ is a graph on $V$ with maximum degree $\Delta(B) \leq \eps n$, then either $H$ admits a perfect $B$-avoiding fractional $T$-tiling or $H$ is $\gamma$-extremal.
\end{lemma}

\begin{proof}[Proof of Lemma~\ref{frac_tiling}]
It suffices to prove the lemma in the case where 5 divides $n$. To see this, form augmented blow-ups $H'$ and $B'$ of $H$ and $B$ in the following way. Both $H'$ and $B'$ have vertex set $V'$, which is a set formed from $V$ by replacing each vertex $u \in V$ by five copies $u_1, \dots, u_5$. So $|V'| = 5n$. The edges of $H'$ are the triples $u_iv_jw_k$ for which $uvw \in E(H)$ and $i,j,k \in [5]$, and also all triples of the form $u_iu_jv_k$ with $u, v \in V$ and $i,j,k \in [5]$ with $i \neq j$. The edges of $B'$ are all the pairs $u_iv_j$ with $uv \in E(B)$ and $i,j \in [5]$, and also all pairs $u_iu_j$ with $u \in V$ and $i, j \in [5]$ with $i \neq j$. This definition implies that $\delta(H') \geq 5\delta(H) \geq (2/5-\alpha) (5n)$ and $\Delta(B') = 5\Delta(B) + 4 \leq 2\eps (5n)$. So the lemma in the case that $5$ divides $n$ implies that either $H'$ admits a perfect $B'$-avoiding fractional $T$-tiling $w'$ or $H'$ is $\gamma$-extremal. In the former case, our choice of $B'$ implies that for each $u \in V$ and $i \neq j$ there is no $T' \in \mathcal{T}(H')$ with $w'(T') > 0$ which contains both $u_i$ and $u_j$. In particular, every $B'$-avoiding copy $T'$ of $T$ in $H'$ corresponds directly to a copy $T^*$ of $T$ in $H$ (by replacing each vertex $u_i$ of $T'$ in $H'$ by the corresponding vertex $u$ in $H$). Setting $w(T^*)$ to be the average of $w'(T')$ over every copy $T'$ of $T$ in $H'$ that corresponds to $T^*$ therefore gives a perfect $B$-avoiding fractional $T$-tiling $w$ in $H$. On the other hand, if $H'$ is $\gamma$-extremal then there is a set $S' \subseteq V'$ of size $3n$ with $d(H'[S']) \leq \gamma$. Let $S^*$ be the set of all vertices $u \in V$ for which $u_i \in S'$ for some $i \in [5]$. We then have $|S'|/5 \leq |S^*| \leq |S'|$ and $e(H[S^*]) \leq e(H'[S'])$, so 
$$d(H[S^*]) = \frac{e(H[S^*])}{\binom{|S^*|}{3}} \leq \frac{e(H'[S'])}{\frac{1}{200}\binom{|S'|}{3}} = 200 d(H'[S']) \leq 200\gamma.$$ 
Since $|S^*| \geq 3n/5$, by an averaging argument there is a set $S \subseteq S^*$ of size $\lfloor 3n/5 \rfloor$ with $d(H[S]) \leq d(H[S^*])$, and $S$ witnesses that $H$ is $200\gamma$-extremal. 

So assume that 5 divides $n$, and let $G$ be the complement of $B$, so each vertex $v \in V(H)$ has at most $\eps n+1$ non-neighbours in $G$. Throughout the remainder of the proof we identify $B$ and $G$ with their edge sets, writing $uv \in B$ and $uv \in G$ for $uv \in E(B)$ and $uv \in E(G)$ respectively.

Let $v_1, \ldots, v_n$ be an enumeration of the vertices of $V(H)$, and let $\mathcal{T}$ be the set of all $T' \in \mathcal{T}(H)$ which do not contain an edge of $B$ (that is, there is no $T' \in \mathcal{T}$ and $uv \in E(B)$ with $u, v \in V(T')$). So $H$ has a perfect $B$-avoiding fractional $T$-tiling if and only if ${\bf 1} \in \pc(\{\mathbbm{1}_S: S \in \mathcal{T}\})$. 

Suppose for a contradiction that $H$ neither is $\gamma$-extremal nor admits a perfect $B$-avoiding fractional $T$-tiling; the latter condition implies that ${\bf 1} \notin \pc(\{\mathbbm{1}_S: S \in \mathcal{T}\})$. So by Farkas' lemma, there exists ${\bf a} \in \mathbb{R}^n$ such that ${\bf a}\cdot {\bf 1} <0$ and ${\bf a} \cdot \mathbbm{1}_S \geq 0$ for every $S \in \mathcal{T}$. By relabelling the vertices of $H$ if necessary, we may assume without loss of generality that the coordinates of $\mathbf{a}$ have $a_1 \leq a_2 \leq \ldots \leq a_n$.

Let $\beta \geq \alpha+5\eps$ be such that $\beta n$ is an integer and $\beta \ll \gamma$.
We consider the partition of $V(H)$ as $V(H) = \bigcup_{X \in \mathcal{V}_1 \cup \mathcal{V}_2 \cup \mathcal{V}_3} X$ where
\begin{enumerate}[(i)]
    \item $\mathcal{V}_1:=\{\{v_i, v_{\beta n+i}, v_{3n/5+\beta n+i}, v_{3n/5+2\beta n+i}, v_{3n/5+3\beta n+i}\}:i \in [\beta n]\}$,
    \item $\mathcal{V}_2:=\{\{v_{3n/5-7\beta n+i}, v_{3n/5-5\beta n+i}, v_{3n/5-3\beta n+i}, v_{3n/5-\beta n+i}, v_{3n/5+4\beta n+i}\}: i \in [2\beta n]\}$, and
    \item $\mathcal{V}_3:=\{\{v_{2\beta n +i}, v_{n/5-\beta n+i}, v_{2n/5-4\beta n +i}, v_{3n/5+6\beta n+i}, v_{4n/5+3\beta n+i}\}:i \in [n/5-3\beta n]\}$.
\end{enumerate}
We will show that there are $T_1, T_2, T_3 \in \mathcal{T}$ such that for each $i \in [3]$ every set $V \in \mathcal{V}_i$ dominates $T_i$. 
It follows that for each $i \in [3]$ we have ${\bf a} \cdot \mathbbm{1}_{T_i} \leq {\bf a} \cdot \mathbbm{1}_{V}$ for every $V \in \mathcal{V}_i$. 
Furthermore, since $\mathcal{V}_1 \cup \mathcal{V}_2 \cup \mathcal{V}_3$ is a partition of $V(H)$, we have $\sum_{i \in [3]} \sum_{V \in \mathcal{V}_i} \mathbbm{1}_V={\bf 1}$. 
So
$$0>{\bf a}\cdot {\bf 1} ={\bf a} \cdot \left( \sum_{i \in [3]} \sum_{V \in \mathcal{V}_i} \mathbbm{1}_V \right) \geq \beta n({\bf a} \cdot \mathbbm{1}_{T_1}) +  2\beta n({\bf a} \cdot \mathbbm{1}_{T_2}) + (n/5-3\beta n)({\bf a} \cdot \mathbbm{1}_{T_3}) \geq 0,$$
a contradiction. We conclude that $H$ either is $\gamma$-extremal or admits a perfect $B$-avoiding fractional $T$-tiling, as required.

It remains to prove the existence of $T_1, T_2$ and $T_3$ as claimed. For $T_1$, let $i =1$ and choose $j \leq \eps n + 2 \leq \beta n$ so that $v_iv_j \in G$, which is possible since $v_1$ has at most $\eps n+1$ non-neighbours in $G$. Next choose $k, \ell \leq 3n/5 + \alpha n + 4\eps n \leq \beta n$ in turn so that $v_k, v_\ell \in N_H(v_iv_j)$ and $v_i, v_j, v_k, v_\ell$ form a clique in $G$, which is possible by the same reason combined with the fact that $\deg_H(v_iv_j) \geq \delta(H) \geq 2n/5-\alpha n$. Finally an analogous calculation shows that we may choose $m \leq 3n/5 + \alpha n + 5 \eps n \leq 3n/5+\beta n$ so that $v_m \in N_H(v_k v_\ell)$ and $v_i,v_j,v_k,v_\ell \in N_G(v_m)$. This gives a copy $T_1$ of $T$ in $H$ with vertices $v_i, v_j, v_k, v_\ell$ and $v_m$ which is $B$-avoiding since the five vertices form a clique in $G$, so $T_1 \in \mathcal{T}$. Moreover, the bounds on $i, j, k, \ell$ and $m$ imply that $T_1$ is dominated by every set in $\mathcal{V}_1$, as required.

We now turn to $T_2$. Let $S = \{v_1, v_2, \dots, v_{3n/5 - 7 \beta n}\}$ and $S' = \{v_1, v_2, \dots, v_{3n/5}\}$. Suppose first that there is no edge $xy \in G[S]$ for which $|N_H(xy) \cap S| > 4 \eps n$. Since there are at most $(\eps n +1)|S|$ pairs $x, y \in S$ with $x \neq y$ and $xy \notin G$, it follows that 
$$e_H(S') \leq (\eps n+1) |S|^2 + \binom{|S|}{2} 4 \eps n + |S' \setminus S| \binom{|S'|}{2} \leq \gamma \binom{3n/5}{3},$$
and so $S'$ witnesses that $H$ is $\gamma$-extremal. So we may assume that there exist distinct $x, y \in S$ with $xy \in G$ and $|N_H(xy) \cap S| > 4 \eps n$. This allows us to choose $z, w \in S$ in turn so that $xyz$ and $xyw$ are both edges of $H$ and so that $x, y, z$ and $w$ form a clique in $G$. Finally, choose $i$ with $i \leq 3n/5+\alpha n + 5 \eps n$ so that $zwv_i$ is an edge of $H$ and so that $x,y,z,w \in N_G(v_i)$. This gives a copy $T_2$ of $T$ in $H$ with vertices $x, y, z, w, v_i$ which is $B$-avoiding, so $T_2 \in \mathcal{T}$. Since $x,y,z,w \in S$ and $i \leq 3n/5+\beta n$ we also have that $T_2$ is dominated by every set in $\mathcal{V}_2$, as required.

Last of all, for $T_3$, let $U = \{v_i : i \leq 5 \eps n\}$, $W = \{v_i : i \leq 2n/5-4\beta n\}$ and $W' = \{v_i : 2n/5-4\beta n+1 \leq i \leq 4n/5\}$. Observe that $G[U]$ must contain a triangle; let $K$ be the set of vertices of this triangle. We consider two cases.

\medskip \noindent {\em Case 1:} there are $x, y \in V(K)$ with $|N_H(xy) \cap W| \geq 3 \eps n$. If so, we may choose $z \in W$ such that $xyz \in E(H)$ and $x, y \in N_G(z)$. Now choose first $v_i \in N_H(x y)$ and then $v_j \in N_H(z v_i)$ such that $i, j \leq 3n/5+\alpha n + 5\eps n$ and so that $x, y, z, v_i$ and $v_j$ form a clique in $G$. This gives a copy $T_3$ of $T$ in $H$ with vertices $x, y, z, v_i$ and $v_j$ which is $B$-avoiding, so $T_3 \in \mathcal{T}$. Since $x, y \in U$, $z \in W$ and $i, j \leq 3n/5+\beta n$ we also have that $T_3$ is dominated by every set in $\mathcal{V}_3$, as required.

\medskip \noindent {\em Case 2:} for all $x, y \in V(K)$ we have $|N_H(xy) \cap W| < 3 \eps n$. In this case, each pair $x, y \in V(K)$ has $|N_H(xy) \cap W'| \geq \delta(H) - 3 \eps n - n/5 \geq n/5 - \alpha n - 3\eps n$. Since $|W'| = 2n/5 + 4\beta n$ it follows that we may write $V(K) = \{x, y, z\}$ in such a way that $|N_H(xy) \cap N_H(yz) \cap W'| \geq 4\eps n$. So we may choose $w \in |N_H(xy) \cap N_H(yz) \cap W'|$ for which $x, y, z \in N_G(w)$. Finally, choose $i \leq 3n/5+\alpha n+5 \eps n$ with $v_i \in N_H(xz)$ and with $z, y, z, w \in N_G(v_i)$. This gives a copy $T_3$ of $T$ in $H$ with vertices $x, y, z, w$ and $v_i$ which is $B$-avoiding, so again $T_3 \in \mathcal{T}$. Since $x, y, z \in U$, $w \in W'$ and $i \leq 3n/5+\beta n$ we also have that $T_3$ is dominated by every set in $\mathcal{V}_3$, as required.
\end{proof}

We now use Lemma~\ref{frac_tiling} to prove another similar result. The difference is that rather than requiring a fractional matching which avoids a set $B$ of pairs of vertices, we now seek a fractional matching which does not put too much weight on any pair of vertices.

\begin{lemma} \label{frac_tiling_with_bound} 
Suppose that $1/n \ll \eps, \alpha \ll \gamma$. Let $H$ be a $3$-graph on $n$ vertices with $\delta(H) \geq 2n/5-\alpha n$. 
If $H$ is not $\gamma$-extremal, then $H$ admits a perfect fractional $T$-tiling in which $w(uv) < 1/(\eps n)$ for all $u, v \in V(H)$.
\end{lemma}

\begin{proof}
For each perfect fractional $T$-tiling $w$ in $H$, define $\psi(w) := \max_{u,v \in V(H)} w(uv)$. Let $W := \min_w \psi(w)$, where the minimum is taken over all perfect fractional $T$-tilings in $H$ (which exists since the set of all perfect fractional $T$-tilings in $H$ is the feasible region of a linear program with integer coefficients), and let $w$ be a perfect fractional $T$-tiling in $H$ which attains this minimum, so $\psi(w) = W$. We may then choose $\mu \in (0, W)$ for which every pair $uv \in \binom{V(H)}{2}$ has either $w(uv) = W$ or $w(uv) < W - \mu$.

Suppose for a contradiction that $W \geq 1/(\eps n)$, and let $B$ be the graph on $V(H)$ consisting of all pairs $uv \in \binom{V(H)}{2}$ with $w(uv) = W$. For each $u \in V(H)$ we then have 
$$W \cdot \deg_B(u) \leq \sum_{v \in V(H) \setminus \{u\}} w(uv) = \sum_{v \in V(H) \setminus \{u\}} \sum_{T' \in \mathcal{T}_{uv}(H)} w(T') = 4 \cdot \sum_{T' \in \mathcal{T}_{u}(H)} w(T') = 4 w(u) = 4,$$ so $\Delta(B) \leq 4/W \leq 4\eps n$. So we may apply Lemma~\ref{frac_tiling} to obtain a perfect  $B$-avoiding fractional $T$-tiling $w'$ in $H$. For each $T' \in \mathcal{T}(H)$ set $w''(T') := (1-\mu) w(T') + \mu w'(T')$. So for every vertex $u \in V(H)$ we have 
$$w''(u) = (1-\mu) w(u) + \mu w'(u) = (1-\mu) + \mu = 1,$$ 
so $w''$ is a perfect fractional $T$-tiling in $H$. Moreover, for each $uv \in E(B)$ we have $w'(uv) = 0$ since $w'$ is $B$-avoiding, and so $w''(uv) = (1-\mu) w(uv) < W$. On the other hand, for each $uv \notin E(B)$ we have $w(uv) < W - \mu$ and so 
$$w''(uv) = (1-\mu)w(uv) + \mu w'(uv) < (1-\mu)(W - \mu) + \mu = W - \mu W + \mu^2 < W.$$
So $\psi(w'') < W$, contradicting our choice of $W$ as the minimum of $\psi(w)$ over all perfect fractional $T$-tilings $w$ in $H$. We conclude that $W < 1/(\eps n)$, as required.
\end{proof}

Theorem~\ref{lem:tiling} follows by combining Lemma~\ref{frac_tiling_with_bound} with the following special case of a theorem of Pippenger and Spencer~\cite{PS}. Note for this that a multi-$k$-graph $G$ is defined identically to a $k$-graph except that the edge set is now a multiset, so edges of $G$ may appear with multiplicity greater than one. Moreover, degrees of vertices and pairs of vertices in $G$ are counted with multiplicity. On the other hand, a matching $M$ in $G$ is still defined to be a set of pairwise-disjoint edges of $G$, so an edge may not appear in $M$ with multiplicity greater than one.

\begin{theorem} \label{regular_to_matching}
Fix $k \geq 2$, suppose that $1/n \ll \delta' \ll \delta, 1/k$, and let $G$ be a multi-$k$-graph on $n$ vertices. If there exists $D \in \mathbb{N}$ for which $(1-\delta') D \leq \deg(u) \leq (1 + \delta') D$ for every $u \in V(G)$ and $\deg(uv) \leq \delta' D$ for all $u, v \in V(G)$, then $G$ admits a matching of size at least $(1-\delta)n/k$.
\end{theorem}

We will use Theorem~\ref{regular_to_matching} in a restated form pertaining to fractional matchings (which are fractional $F$-tilings in the particular case which $F$ is the $k$-graph with $k$ vertices and one edge). For this, let $H$ be a $k$-graph on $n$ vertices, and observe that the set of fractional matchings in $H$ is the feasible region of a linear program with integer coordinates, which is a convex polytope all of whose vertices are rational. It follows that for any fixed $\eps \in \mathbb{R}$, if $H$ admits a fractional matching in which $w(u)\geq 1-\eps$ for every $u \in V(H)$ and $w(uv) \leq \eps$ for all $u, v \in V(H)$, then $H$ admits a fractional matching with this property in which $w(e)$ is rational for every $e \in E(H)$. So we may choose $D \in \mathbb{N}$ to be a common denominator of all of the edge weights in $w$, meaning that for each $e$ we have $w(e) = m(e)/D$ for some integer $m(e)$ with $0 \leq m(e) \leq D$. The multi-$k$-graph $G$ with vertex set $V(G) := V(H)$ which includes each edge $e \in H$ with multiplicity $m(e)$ then satisfies $(1-\eps)D \leq \deg_G(u) \leq D$ for every $u \in V(H)$ and $\deg_G(uv) \leq \eps D$ for all $u, v \in V(H)$. Using Theorem~\ref{regular_to_matching} with $\eps$ and $\eta$ in place of $\delta'$ and $\delta$ respectively, we then obtain the following restatement for fractional matchings.

\begin{corollary} \label{frac_to_almost}
Fix $k \geq 2$ and suppose that $1/n \ll \eps \ll \eta, 1/k$, and let $H$ be a $k$-graph on $n$ vertices. If $H$ admits a fractional matching in which $w(u) \geq 1-\eps$ for every $u \in V(H)$ and $w(uv) \leq \eps$ for all $u, v \in V(H)$, then $H$ admits a matching of size at least $(1-\eta) n/k$.
\end{corollary}

At last we present the proof of the tiling lemma.

\begin{proof}[Proof of Lemma~\ref{lem:tiling}]
Introduce a new constant $\eps$ with $1/n \ll \eps \ll \eta$. We may assume that $H$ is not $\gamma$-extremal, and so $H$ admits a perfect fractional $T$-tiling $w$ with $w(uv) < 1/\eps n$ for all $u, v \in V(H)$ by Lemma~\ref{frac_tiling_with_bound}. Let $G$ be the $5$-graph with vertex set $V(G) = V(H)$ whose edges are all sets $S \in \binom{V(H)}{5}$ which support a copy of $T$ in $H$. For each $S \in \binom{V(H)}{5}$, set $w'(S)$ to be the sum of $w(T')$ over all copies $T'$ of $T$ in $H[S]$. This gives a perfect fractional matching $w'$ in $G$ with $w'(uv) = w(uv) < 1/\eps n$ for all $u, v \in V(H)$, so, by Corollary~\ref{frac_to_almost}, $G$ admits a matching $M = \{S_1, S_2, \dots, S_r\}$ of size at least $(1-\eta) n/5$. For each $i \in [r]$ let $T_i$ be a copy of $T$ in $H$ with vertex set $S_i$; then $\{T_1, T_2, \dots, T_r\}$ is an $T$-tiling in $H$ which covers at least $(1-\eta) n$ vertices of $H$, as required.
\end{proof}

We conclude with a remark on the methods applied in this section. The use of Farkas' lemma to obtain a perfect fractional tiling in a $k$-graph $H$, which is then used to obtain an almost-perfect (non-fractional) tiling, follows the approach of several previous works in this area. However, we would highlight the case distinction in the proof of Lemma~\ref{frac_tiling_with_bound} as being a novel element. Through these cases our choice of copies of $T$ which are dominated by sets of the partition of $V(H)$ was dependent on the structure of $H$, specifically the intersection of neighbourhoods of triples of vertices in $\{v_1, v_2, \dots, v_{5\eps n}\}$ which form a triangle in $G$. To our knowledge previous arguments using Farkas' lemma have not used structural information about $H$ in this way. We believe that more sophisticated developments of this approach may enable future work to discover currently unknown minimum degree thresholds for tilings with other $k$-graphs.

\section{Perfect rainbow tilings} \label{sec:rainbow} We conclude by presenting the short derivation of Theorem~\ref{thm:rainbow}. 
The key result which we use for this is a theorem given by Lang~\cite{Lang} which relates the following quantities. 
The \emph{tiling threshold} $\delta_\mathrm{t}$ is the infimum of all $\delta \in [0,1]$ such that for all $\mu > 0$ there exists $n_0$ such that for all $n \geq n_0$ with $5 \mid n$ every $3$-graph $H$ on $n$ vertices with $\delta(H) \geq (\delta + \mu)n$ admits a perfect $T$-tiling. 
Similarly, the \emph{rainbow tiling threshold} $\delta_\mathrm{r}$ is the infimum of all $\delta \in [0,1]$ such that for all $\mu > 0$ there exists $n_0$ such that for all $n \geq n_0$ with $5 \mid n$ every collection $\mathcal{H} = \{H_1, \dots, H_{3n/5}\}$ of $3$-graphs on a common vertex set of size $n \geq n_0$ with $\delta(H_i) \geq (\delta + \mu)n$ for every $i \in [3n/5]$ admits a perfect rainbow $T$-tiling. 
Finally, for $3$-graphs $H_1$ and $H_2$ on a common vertex set $V$, a {\it colour covering homomorphism} from $T$ to $(H_1, H_2)$ is an injective function $\phi: V(T) \rightarrow V$ such that, for some $e \in E(T)$, we have $\phi(e) \in E(H_1)$ and $\phi(e') \in E(H_2)$ for each edge $e' \in E(T)$ with $e' \neq e$. 
The \emph{colour covering threshold} $\delta_\mathrm{c}$ is the infimum of all $\delta \in [0,1]$ such that for all $\mu > 0$ there exists $n_0$ such that for all $n \geq n_0$ every pair $H_1$, $H_2$ of $3$-graphs on a common vertex set of size $n$ with $\delta(H_1), \delta(H_2) \geq (\delta + \mu)n$ admits a colour covering homomorphism from $T$ to $(H_1, H_2)$. We will apply the following special case of a theorem given by Lang. While we state this theorem (and the preceding definitions) only for minimum codegree thresholds and the generalised triangle 3-graph~$T$, the full version of the theorem gives the same statement for all $k$-graphs and for minimum $d$-degree conditions (which we do not define here) for all $d \in [k-1]$.

\begin{theorem}\cite[Theorem~5.11]{Lang} $\delta_{\mathrm{r}} = \max(\delta_{\mathrm{c}}, \delta_{\mathrm{t}})$ \label{thm:lang}
\end{theorem}

\begin{proof}[Proof of Theorem~\ref{thm:rainbow}]
Let $H_1, H_2$ be graphs on a common vertex set $V$ of size $n \geq 5$ with $\delta(H_1), \delta(H_2) \geq 3$. Fix an edge $xyz \in E(H_1)$. Since $\delta(H_2) \geq 3$ we may then choose $u, v \in V \setminus \{x,y,z\}$ in turn with $xyu \in E(H_2)$ and $uvz \in E(H_2)$. The edges $xyz, xyu$ and $uvz$ then form a copy of $T$ in $H_1 \cup H_2$ with $xyz$ being an edge of $H_1$ and $xyu, uvz$ being edges of $H_2$, giving rise to a colour covering homomorphism from $T$ to $(H_1, H_2)$. It follows from the definition of $\delta_\mathrm{c}$ that $\delta_\mathrm{c} = 0$. 

Theorem~\ref{thm:main}, together with the subseqent construction $H_\mathrm{ext}$ which demonstrates that the minimum codegree condition of Theorem~\ref{thm:main} is best possible, shows that $\delta_t = 2/5$. So by 
Theorem~\ref{thm:lang}, we have $\delta_r = 2/5$, and by definition of~$\delta_r$ this completes the proof.
\end{proof}

\end{document}